\documentclass[12pt,leqno]{article}
\usepackage{cite}
\usepackage{amsfonts}
\usepackage{listings}
\usepackage{fullpage}
\usepackage{amssymb}
\usepackage{amsmath}
\usepackage{amsthm}
\newtheorem{theorem}{Theorem}[section]

\newtheorem{lemma}{Lemma}[section]
\newtheorem{remark}{Remark}[section]

\begin{document}

\title{\large\bf On the cardinality of general $h$-fold sumsets
  \footnote{This work was supported by the National
Natural Science Foundation of China, Grant No. 11371195.}}
\date{}
\author{Quan-Hui Yang$^{1,2}$ and Yong-Gao Chen$^2$
\footnote{Emails:~yangquanhui01@163.com (Q.-H Yang), ygchen@njnu.edu.cn (Y.-G Chen).}\\
\small 1. School of Mathematics and Statistics, \\ \small Nanjing
University of Information Science and Technology,
\\ \small    Nanjing 210044, P. R. CHINA
\\ \small 2. School of Mathematical Sciences and Institute of
Mathematics,
\\ \small Nanjing Normal University,  Nanjing 210023, P. R. CHINA}
\maketitle
\vskip 3mm

\begin{abstract} Let $A=\{a_0,a_1,\ldots,a_{k-1}\}$ be a  set of $k$ integers. For any integer
$h\ge 1$ and any ordered $k$-tuple of positive integers
$\mathbf{r}=(r_0,r_1,\ldots,r_{k-1})$, we define a general
$h$-fold sumset, denoted by $h^{(\mathbf{r})}A$, which is the set
of all sums of $h$ elements of $A$, where $a_i$ appearing in the
sum can be repeated at most $r_i$ times for $i=0,1,\ldots,k-1$. In
this paper, we give the best lower bound for $|h^{(\mathbf{r})}A|$
in terms of $\mathbf{r}$ and $h$ and determine the structure of
the set $A$ when $|h^{(\mathbf{r})}A|$ is minimal. This
generalizes results of Nathanson, and recent results of Mistri and
Pandey and also solves a problem of Mistri and Pandey.



\end{abstract}

\section{Introduction}\label{b}
Let $\mathbb{N}$ denote the set of all nonnegative integers. For
any
 finite set of integers $A$ and any positive integer $h\ge
2$, define
$$hA=\{a_1+a_2+\cdots+a_h:a_i\in A (1\le i\le h)\}$$
and
$$h\text{\^{}}A=\{a_1+a_2+\cdots+a_h:a_i\in A (1\le i\le h), ~a_i\not=a_j~\text{for all }~i\not=j\}.$$
Sumsets are important in additive number theory (see \cite{Kapoor,lev1,lev2,Nathanson72,Nathanson89,Pan,chen,yang}).

Finding lower bounds for $|hA|$ and $|h\text{\^{}}A|$ in terms of
$h$ and $|A|$ and determining the structure of sets $A$ for which
$|hA|$ or $|h\text{\^{}}A|$ are minimal are important problems in
additive number theory.

 Nathanson \cite{Nathanson96} proved the following fundamental and important results.

\noindent{\bf  Theorem A.} (See \cite[Theorem 1.3]{Nathanson96})
{\em~ Let $h\ge 2$ be an integer and  $A$ a  finite set of
integers with $|A|=k$. Then
$$|hA|\ge hk-h+1.$$}

\noindent{\bf  Theorem B.} (See \cite[Theorem 1.6]{Nathanson96})
{\em~ Let $h\ge 2$ be an integer and  $A$ a  finite set of
integers with $|A|=k$. Then
$$|hA|= hk-h+1$$ if and only if $A$ is a $k$-term arithmetic progression.}

\noindent{\bf  Theorem C.} (See \cite[Theorem 1.9]{Nathanson96} or
\cite[Theorem 1]{Nathanson95}) {\em~ Let $A$ be a  finite set of
integers with $|A|=k$ and let $1\le h\le k$. Then
$$|h\text{\^{}}A|\ge hk-h^2+1.$$ This lower bound is best possible.}

\noindent{\bf  Theorem D.} (See \cite[Theorem 1.10]{Nathanson96} or \cite[Theorem 2]{Nathanson95})
{\em~ Let $k\ge 5$ and let $2\le h\le k-2$. If $A$ is a set of $k$ integers such that
$$|h\text{\^{}}A|= hk-h^2+1,$$ then $A$ is a $k$-term arithmetic progression.}

From now on, we assume that $A=\{ a_0, a_1, \dots , a_{k-1}\} $ is
a set of integers with $a_0<a_1<\cdots <a_{k-1}$. For two positive
integers $h$ and $r$, define
$$h^{(r)}A=\left\{\sum_{i=0}^{k-1}s_ia_i:0\le s_i\le r
~\text{for}~i=0,1,\ldots,k-1~\text{and}~\sum_{i=0}^{k-1}s_i=h\right\}.$$
Clearly, $h^{(1)}A=h\text{\^{}}A$ and $h^{(h)}A=hA$. Recently,
Mistri and Pandey generalized the above results.

\noindent{\bf  Theorem E.} (See \cite[Theorem 2.1]{Pandey}) {\em~
Let $A$ be a  set of $k$ integers, $r$ and $h$ be two integers
such that $1\le r\le h\le rk$.  Then
$$|h^{(r)}A|\ge mr(k-m)+(h-mr)(k-2m-1)+1,$$
where $m$ is the integer with $h/r-1<m\le h/r$. This lower bound
is best possible.}

\noindent{\bf  Theorem F.} (See \cite[Theorem 3.1, Theorem
3.2]{Pandey}) {\em~ Let $k\ge 3$, $r$ and $h$ be integers with
$1\le r\le h\le rk-2$ and $(k, h, r)\not= (4,2,1)$. If $A$ is a
set of $k$ integers such that
$$|h^{(r)}A|=mr(k-m)+(h-mr)(k-2m-1)+1,$$
where $m$ is the integer with $h/r-1<m\le h/r$, then $A$ is a
$k$-term arithmetic progression.}

For any ordered $k$-tuple of positive  integers
$\mathbf{r}=(r_0,r_1,\ldots,r_{k-1})$ and  any positive integer
$h$, define
$$h^{(\mathbf{r})}A=\left\{\sum_{i=0}^{k-1}s_ia_i:0\le s_i\le r_i
(0\le i\le k-1), ~\sum_{i=0}^{k-1}s_i=h\right\}.$$ Clearly, if
$\mathbf{r}=(r,r,\ldots,r)$ is an ordered $k$-tuple of positive
integers, then $h^{(\mathbf{r})}A=h^{(r)}A$.

Mistri and Pandey \cite[Concluding Remarks]{Pandey} said that it
is interesting to study the direct and inverse problems related to
sumset $h^{(\mathbf{r})}A$.

In this paper, we solve this problem.

For convenience, let $\sum\limits_{x=a}^{b}f(x)=0$ if $a>b$. Let
$I_{\mathbf{r}}(h)$ be the largest integer and $M_{\mathbf{r}}(h)$
be the least integer such that
\begin{eqnarray*}\sum_{j=0}^{I_{\mathbf{r}}(h)-1}r_j\le h,
\quad \sum_{j=M_{\mathbf{r}}(h)+1}^{k-1}r_j\le h,\end{eqnarray*}
and let
\begin{eqnarray*}\delta_{\mathbf{r}}(h)
=h-\sum_{j=0}^{I_{\mathbf{r}}(h)-1}r_j, \quad
\theta_{\mathbf{r}}(h)
=h-\sum_{j=M_{\mathbf{r}}(h)+1}^{k-1}r_j.\end{eqnarray*} Let
$$L(\mathbf{r}, h)=\sum_{j=M_{\mathbf{r}}(h)+1}^{k-1}j r_j
-\sum_{j=0}^{I_{\mathbf{r}}(h)-1}j r_j+M_{\mathbf{r}}(h)
\theta_{\mathbf{r}(h)}-I_{\mathbf{r}}(h)\delta_{\mathbf{r}(h)}+1.$$

In this paper, we prove the following theorems.
\begin{theorem}\label{thm1} Let $A=\{ a_0, a_1, \dots , a_{k-1}\} $ be
a set of integers with $a_0<a_1<\cdots <a_{k-1}$,
$\mathbf{r}=(r_0,r_1,\ldots,r_{k-1})$ be an ordered $k$-tuple of
positive integers and $h$ be an integer with $$2\le h\le
\sum_{j=0}^{k-1}r_j.$$ Then $$|h^{(\mathbf{r})}A|\ge L(\mathbf{r},
h).$$ This lower bound is best possible.
\end{theorem}

\begin{theorem}\label{thm2} Let $k\ge 5$  be an integer, $\mathbf{r}=(r_0,r_1,\ldots,r_{k-1})$
be an ordered $k$-tuple of positive integers and let $h$ be an
integer with $$2\le h\le \sum_{j=0}^{k-1}r_j-2.$$ If $A$ is a set
of $k$ integers, then
\begin{eqnarray*}|h^{(\mathbf{r})}A|=L(\mathbf{r}, h)
\end{eqnarray*}
if and only if  $A$ is a $k$-term arithmetic progression.
\end{theorem}

\begin{remark}For Theorem \ref{thm2} with $1\le k\le 4$,
we shall give  complete results in Section \ref{secx}. Since
$$L( (r, r, \dots , r), h)=mr(k-m)+(h-mr)(k-2m-1)+1,$$ Theorem F is
a corollary of Theorem \ref{thm2} and Theorems \ref{thm3} and
\ref{thm4} in Section \ref{secx}.
\end{remark}

 \begin{remark} If $h=1$, then $h^{(\mathbf{r})}A = A$. So
$|h^{(\mathbf{r})}A|=k$.

If $$h=\sum_{j=0}^{k-1}r_j-1,$$ then $$h^{(\mathbf{r})}A =\left\{
\sum_{j=0}^{k-1}r_j a_j - a_i : 0\le i\le k-1 \right\} .$$ So
$|h^{(\mathbf{r})}A|=k$.

If $$h=\sum_{j=0}^{k-1}r_j,$$ then
$$h^{(\mathbf{r})}A =\left\{ \sum_{j=0}^{k-1}r_j a_j \right\} .$$ So $|h^{(\mathbf{r})}A|=1$.
\end{remark}

\section{Proofs}

 For any $k$-tuple
$X=(x_0,x_1,\ldots,x_{k-1})\in \mathbb{N}^k$, define the function
$$\phi_A(X)=\sum_{j=0}^{k-1}{x_ja_j}.$$

For any ordered $k$-tuple of positive integers
$\mathbf{r}=(r_0,r_1,\ldots,r_{k-1})$ and any positive integer
$h$, let $R(\mathbf{r},h)$ be the set of all ordered $k$-tuple
 $(x_0,x_1,\ldots,x_{k-1})$ of $\mathbb{N}^k$ such that
\begin{eqnarray*}
\sum_{j=0}^{k-1}x_j=h,\quad  0\le x_i\le r_i, \quad
i=0,1,\ldots,k-1.
\end{eqnarray*}
Then
 $$h^{(\mathbf{r})}A=\{\phi_A(X):X\in R(\mathbf{r},h)\}.$$
For any positive integer $k$ and any $k$-tuple
$X=(x_0,x_1,\ldots,x_{k-1})\in \mathbb{N}^k$, define
 the weighted
sum $$S(X)=\sum_{j=0}^{k-1} {j x_j}.$$ For two $k$-tuples
$U=(u_0,u_1,\ldots,u_{k-1}),W=(w_0,w_1,\ldots,w_{k-1})\in
\mathbb{N}^k$,
we call $U\rightarrow W$ a {\em step} if there exists an index
$j\ge 0$ such that $w_j=u_j-1$, $w_{j+1}=u_{j+1}+1$ and $w_i=u_i$
for all integers $i\not=j,j+1$. We call $X_1\rightarrow
X_2\rightarrow \cdots \rightarrow X_t$ a {\em $(\mathbf{r},
h)$-path} of length $t$, if $X_i\in R(\mathbf{r},h) (1\le i\le t)$
and $X_{i+1}\rightarrow X_{i} (1\le i\le t-1)$ are steps. It is
clear that if $X_1\rightarrow X_2\rightarrow \cdots \rightarrow
X_t$ is a $(\mathbf{r}, h)$-path of length $t$, then
$$S(X_{i+1})-S(X_{i})= 1 (1\le i\le t-1).$$
Thus $S(X_t)-S(X_1)=t-1$.

Let
$$V=(r_0,r_1,\ldots,r_{I_{\mathbf{r}}(h)-1},\delta_{\mathbf{r}(h)},0,\ldots,0)$$
and
$$V'=(0,\ldots,0,\theta_{\mathbf{r}(h)},r_{M_{\mathbf{r}}(h)+1},\ldots,r_{k-1}),$$
where
$I_{\mathbf{r}}(h),\delta_{\mathbf{r}(h)},\theta_{\mathbf{r}(h)},M_{\mathbf{r}}(h)$
are defined as in Section \ref{b}. Then $V, V'\in
R(\mathbf{r},h)$.

\begin{lemma}\label{lem0}We have $S(V')-S(V)+1=L(\mathbf{r}, h)$. In particular, any $(\mathbf{r}, h)$-path from $V$
to $V'$ has length $L(\mathbf{r}, h)$.\end{lemma}

\begin{proof} Noting that
$$S(V)=\sum_{j=0}^{I_{\mathbf{r}}(h)-1}j r_j+I_{\mathbf{r}}(h)\delta_{\mathbf{r}(h)},\quad
S(V')=\sum_{j=M_{\mathbf{r}}(h)+1}^{k-1}j
r_j+M_{\mathbf{r}}(h)\theta_{\mathbf{r}(h)},$$ we have
\begin{eqnarray*}S(V')-S(V)&=&\sum_{j=M_{\mathbf{r}}(h)+1}^{k-1}j r_j
-\sum_{j=0}^{I_{\mathbf{r}}(h)-1}j
r_j+M_{\mathbf{r}}(h)\theta_{\mathbf{r}(h)}-I_{\mathbf{r}}(h)\delta_{\mathbf{r}(h)}\\
&=&L(\mathbf{r}, h)-1.\end{eqnarray*} Since a $(\mathbf{r},
h)$-path from $V$ to $V'$ has length $S(V')-S(V)+1$, it follows
that any $(\mathbf{r}, h)$-path from $V$ to $V'$ has length
$L(\mathbf{r}, h)$.
\end{proof}

\begin{lemma}\label{lem1} Let
$X=(x_0,x_1,\ldots,x_{k-1})\in R(\mathbf{r},h)$ and $
Y=(y_0,y_1,\ldots,y_{k-1})\in R(\mathbf{r},h)$ with $X\not=Y$. If
$$\sum_{j=i}^{k-1}x_j\le \sum_{j=i}^{k-1}y_j, \quad
i=1,2,\ldots,k-1,$$ then there exists a $(\mathbf{r}, h)$-path
from $X$ to $Y$.
\end{lemma}

\begin{proof} Let $X_0=X\rightarrow X_1\rightarrow \cdots\rightarrow X_g$ be a
 $(\mathbf{r}, h)$-path of the maximal length such that
\begin{eqnarray}\label{condition}\sum_{j=t}^{k-1}x_{i,j}\le \sum_{j=t}^{k-1}y_j,
\quad ~1\le t\le k-1, 1\le i\le g,\end{eqnarray} where
$X_i=(x_{i,0},x_{i,1},\ldots,x_{i,k-1})~(0\le i\le g)$. Now we
prove that $X_g=Y$. Suppose that $X_g\not=Y$. Let $s$ be the
maximal index with $x_{g,s}\not=y_s$. Noting that $X, Y\in
R(\mathbf{r},h)$, we have
$$\sum_{j=0}^{k-1}x_{g,j}=h=\sum_{j=0}^{k-1}y_j.$$
Hence $s\ge 1$. Since
$$\sum_{j=s}^{k-1}x_{g,j}\le \sum_{j=s}^{k-1}y_j,$$ it follows
from the definition of $s$ that  $x_{g,s}<y_s$. If $x_{g,s-1}>0$,
let
$$X_{g+1}=(x_{g,0},\ldots,x_{g,s-1}-1,x_{g,s}+1,x_{g,s+1},\ldots,x_{g,k-1}),$$
then $X_g\rightarrow X_{g+1}$ is a $(\mathbf{r}, h)$-path and
$X_{g+1}$ also satisfies \eqref{condition}. This is a
contradiction with the maximality of $g$. Hence $x_{g,s-1}=0$. If
$x_{g,j}=0$ for all $0\le j\le s-1$, then
\begin{eqnarray*}\sum_{j=0}^{k-1}x_{g,j}&=&x_{g,s}+\sum_{j=s+1}^{k-1}x_{g,j}\\
&=&x_{g,s}+\sum_{j=s+1}^{k-1}y_j\\
&<&y_s+\sum_{j=s+1}^{k-1}y_j\\
&\le &\sum_{j=0}^{k-1}y_j=h,\end{eqnarray*} a contradiction with
$X_g\in R(\mathbf{r},h)$ (see the definition of $(\mathbf{r},
h)$-path). Thus there exists an index $j$ with $0\le j< s-1$ such
that $x_{g,j}>0$. We assume that $j$ is the largest such index.
Let
$$X_{g+1}=(x_{g,0},\ldots,x_{g,j}-1,x_{g,j+1}+1,0, \dots , 0, x_{g,s},\ldots,x_{g,k-1}).$$
Then $X_g\rightarrow X_{g+1}$ is a $(\mathbf{r}, h)$-path. Since
$X_{g}$  satisfies \eqref{condition}, it follows that $X_{g+1}$
also satisfies \eqref{condition}. This is a contradiction with the
maximality of $g$. Therefore, $X_g=Y$.
\end{proof}

\begin{lemma}\label{lem2}  Let $X_1\rightarrow
X_2\rightarrow \cdots\rightarrow X_{t-1}\rightarrow X_t$ and
$X_1\rightarrow X_2'\rightarrow \cdots\rightarrow
X_{t-1}'\rightarrow X_t$ be two different $(\mathbf{r}, h)$-paths
from $X_1$ to $X_t$. If $A$ is a set of $k$ integers such that
$|h^{(\mathbf{r})}A|=L(\mathbf{r}, h)$, then
$\phi_A(X_i)=\phi_A(X_i')$ for $i=2,3,\ldots,t-1$.
\end{lemma}

\begin{proof}  By Lemma \ref{lem1}, there exists a
$(\mathbf{r}, h)$-path from $V$ to $X_1$ and another $(\mathbf{r},
h)$-path from $X_t$ to $V'$. Thus we have the following
$(\mathbf{r}, h)$-path from $V$ to $V'$:
\begin{eqnarray}\label{path}V\rightarrow \cdots\rightarrow X_1\rightarrow X_2\rightarrow \cdots\rightarrow X_{t-1}
\rightarrow X_t\rightarrow \cdots \rightarrow V'.\end{eqnarray} By
Lemma \ref{lem0}, the length of
 the $(\mathbf{r}, h)$-path \eqref{path} is
$L(\mathbf{r}, h)=|h^{(\mathbf{r})}A|$. Clearly,
$$\phi_A(V)<\cdots<\phi_A(X_1)<\phi_A(X_2)<\cdots<\phi_A(X_{t-1})<\phi_A(X_{t})<\cdots<\phi_A(V').$$
Since
$$\{\phi_A(X):X ~\text{is on the $(\mathbf{r}, h)$-path \eqref{path}}\}\subseteq h^{(\mathbf{r})}A$$
and
$$|\{\phi_A(X):X ~\text{is on the $(\mathbf{r}, h)$-path \eqref{path}}\}|= |h^{(\mathbf{r})}A|,$$
it follows that $$h^{(\mathbf{r})}A=\{\phi_A(X):X ~\text{is on the
$(\mathbf{r}, h)$-path \eqref{path}}\}.$$ Noting that
$$\{\phi_A(X_2'),\phi_A(X_3'),\ldots,\phi_A(X_{t-1}')\}\subseteq h^{(\mathbf{r})}A$$
and
$$\phi_A(X_1)<\phi_A(X_2')<\cdots<\phi_A(X_{t-1}')<\phi_A(X_{t}),$$
we have $\phi_A(X_i)=\phi_A(X_i')$ for $i=2,3,\ldots,t-1$.
\end{proof}

\begin{lemma}\label{lem3} Let $c_i$ and $d_i (0\le i\le k-1)$ be integers with $c_i\le d_i(0\le i\le k-1)$.
If $h$ is an integer with $$\sum_{i=0}^{k-1} c_i \le h\le
\sum_{i=0}^{k-1} d_i,$$ then there exist integers $x_i (0\le i\le
k-1)$  with $c_i\le x_i\le  d_i(0\le i\le k-1)$ such that
$$h=x_0+x_1+\cdots +x_{k-1}.$$
\end{lemma}

Proof is left to the reader.

\begin{proof}[Proof of Theorem \ref{thm1}]
 By Lemma
\ref{lem1}, there exists a $(\mathbf{r}, h)$-path
$V=V_0\rightarrow V_1\rightarrow \cdots \rightarrow V_{\ell}=V'$.
By Lemma \ref{lem0}, we have $\ell +1=L(\mathbf{r}, h)$. Since
$\phi_A(V_i)\in h^{(\mathbf{r})}A (0\le i\le \ell )$ and
$\phi_A(V_{i+1})>\phi_A(V_i) (0\le i\le \ell -1 )$, we have
\begin{equation}\label{w}|h^{(\mathbf{r})}A|\ge \ell+1=L(\mathbf{r},
h).\end{equation}

Next we show that this lower bound is optimal. Let $A=\{0,1,\ldots,k-1\}$.
Then the smallest integer in $h^{(\mathbf{r})}A$ is
\begin{eqnarray*}&&\underbrace{0+\cdots+0}_{r_0~\text{copies}}+\underbrace{1+\cdots+1}_{r_1~\text{copies}}
+\cdots+\underbrace{(I_{\mathbf{r}}(h)-1)+\cdots+(I_{\mathbf{r}}(h)-1)}_{r_{I_{\mathbf{r}}(h)-1}~\text{copies}}\\
&&+\underbrace{I_{\mathbf{r}}(h)+\cdots+I_{\mathbf{r}}(h)}_{\delta_{\mathbf{r}(h)}~\text{copies}}\\
&=& S(V)
\end{eqnarray*}
and the largest integer in $h^{(\mathbf{r})}A$ is
\begin{eqnarray*}&&\underbrace{M_{\mathbf{r}}(h)+\cdots+M_{\mathbf{r}}(h)}_{\theta_{\mathbf{r}}(h)~\text{copies}}
+\underbrace{(M_{\mathbf{r}}(h)+1)+\cdots+(M_{\mathbf{r}}(h)+1)}_{r_{M_{\mathbf{r}}(h)+1}~\text{copies}}\\
&&+\cdots+\underbrace{(k-2)+\cdots+(k-2)}_{r_{k-2}~\text{copies}}
+\underbrace{(k-1)+\cdots+(k-1)}_{r_{k-1}~\text{copies}}\\
&=& S(V').
\end{eqnarray*}
It follows that
$$h^{(\mathbf{r})}A\subseteq [S(V),
S(V')].$$ Thus, by Lemma \ref{lem0}, we have
\begin{equation}\label{s}|h^{(\mathbf{r})}A|\le  S(V')-S(V)+1=L(\mathbf{r}, h).\end{equation} By
\eqref{w} and \eqref{s}, we have
$$|h^{(\mathbf{r})}A|=L(\mathbf{r}, h).$$
\end{proof}

\begin{proof}[Proof of Theorem \ref{thm2}]

Suppose that $|h^{(\mathbf{r})}A|=L(\mathbf{r}, h)$. For any
integer $j$ with $0\le j\le k-4$, by $$2\le h\le
\sum_{i=0}^{k-1}r_i-2$$ and Lemma \ref{lem3}, there exists
$$X=(x_0,
x_1, \dots, x_j, x_{j+1}, x_{j+2}, x_{j+3}, \dots , x_{k-1}) \in
R(\mathbf{r},h)$$ such that
$$1\le x_j\le r_j, \quad 0\le x_{j+1}\le r_{j+1}-1,\quad 1\le x_{j+2}\le
r_{j+2},\quad 0\le x_{j+3}\le r_{j+3}-1.$$ Then
\begin{eqnarray*}&&(\dots, x_j, x_{j+1}, x_{j+2}, x_{j+3}, \dots
)\\
&\rightarrow& (\dots, x_j-1, x_{j+1}+1, x_{j+2}, x_{j+3}, \dots
)\\
&\rightarrow& (\dots, x_j-1, x_{j+1}+1, x_{j+2}-1, x_{j+3}+1,
\dots ) \end{eqnarray*} and
\begin{eqnarray*}&&(\dots, x_j, x_{j+1}, x_{j+2}, x_{j+3}, \dots   )\\
&\rightarrow & (\dots, x_j, x_{j+1}, x_{j+2}-1, x_{j+3}+1, \dots
)\\
& \rightarrow& (\dots, x_j-1, x_{j+1}+1, x_{j+2}-1, x_{j+3}+1,
\dots )\end{eqnarray*} are two different $(\mathbf{r}, h)$-paths.
By Lemma \ref{lem2}, we have
$$\phi_A ((\dots, x_j-1, x_{j+1}+1, x_{j+2}, x_{j+3}, \dots
)) = \phi_A ((\dots, x_j, x_{j+1}, x_{j+2}-1, x_{j+3}+1, \dots
)).$$ This implies that $a_{j+1}-a_j=a_{j+3}-a_{j+2}$. Therefore,

$$a_{1}-a_0=a_{3}-a_{2}=a_{5}-a_{4}=\cdots , \quad a_{2}-a_{1}=a_{4}-a_{3}=a_6-a_5=\cdots.$$
In order to prove that $A$ is a $k$-term arithmetic progression,
it suffices to prove $a_4-a_3=a_1-a_0$.

By $$2\le h\le \sum_{i=0}^{k-1}r_i-2$$ and Lemma \ref{lem3}, there
exists
$$Y=(y_0,
y_1, y_2, y_3, y_4, \dots, y_{k-1}) \in R(\mathbf{r},h)$$ such
that
$$1\le y_0\le r_0, \quad 0\le y_1\le r_1-1,
\quad 1\le y_3\le r_3, \quad 0\le y_4\le r_4-1.$$ Then
\begin{eqnarray*}&& (y_0,
y_1, y_2, y_3, y_4, \dots, y_{k-1}) \\
&\rightarrow & (y_0-1,
y_1+1, y_2, y_3, y_4, \dots, y_{k-1}) \\
&\rightarrow &(y_0-1, y_1+1, y_2, y_3-1, y_4+1, \dots, y_{k-1})
\end{eqnarray*}
and
\begin{eqnarray*}&& (y_0,
y_1, y_2, y_3, y_4, \dots, y_{k-1}) \\
&\rightarrow & (y_0,
y_1, y_2, y_3-1, y_4+1, \dots, y_{k-1}) \\
&\rightarrow &(y_0-1, y_1+1, y_2, y_3-1, y_4+1, \dots, y_{k-1})
\end{eqnarray*}
are two different $(\mathbf{r}, h)$-paths. By Lemma \ref{lem2}, we
have
$$\phi_A ((y_0-1,
y_1+1, y_2, y_3, y_4, \dots, y_{k-1})) = \phi_A ((y_0, y_1, y_2,
y_3-1, y_4+1, \dots, y_{k-1})).$$ This implies that
$a_1-a_0=a_4-a_3$.

Therefore, $A$ is a $k$-term arithmetic progression.

Conversely, if $A$ is a $k$-term arithmetic progression, without
loss of generality, we may assume that $A=\{0,1,\dots , k-1\}$. By
the proof of Theorem \ref{thm1}, we have
$|h^{(\mathbf{r})}A|=L(\mathbf{r}, h)$.
\end{proof}

\section{Cases $1\le k\le 4$} \label{secx}

For $k=1$ and $1\le h\le r_0$, it is easy to see that
$h^{(\mathbf{r})}A = \{ h a_0 \}$. So $|h^{(\mathbf{r})}A|=1$.

For $k=2$ and $1\le h\le r_0+r_1$, we have
$$h^{(\mathbf{r})}A =\{ x_0 a_0 +x_1 a_1 : 0\le x_0\le r_0, 0\le
x_1\le r_1, x_0+x_1=h, x_0,x_1\in \mathbb{N}\} .$$ So
$$|h^{(\mathbf{r})}A|= |\{ (x_0, x_1): 0\le x_0\le r_0, 0\le
x_1\le r_1, x_0+x_1=h, x_0,x_1\in \mathbb{N}\} |.$$

Now we deal with the cases $k=3$ and $k=4$.

\begin{theorem}\label{thm3} Let $A=\{ a_0<a_1<a_2\}$ be a set of integers and
$\mathbf{r}= (r_0,r_1,r_2)$ be an ordered $3$-tuple of positive
integers. Suppose that $h$ is an integer with $2\le h\le
r_0+r_1+r_2 -2$. Then

(i) for $r_1=1$, we have $|h^{(\mathbf{r})}A|=L(\mathbf{r}, h)$;

(ii) for $r_1\ge 2$, we have $|h^{(\mathbf{r})}A|=L(\mathbf{r},
h)$ if and only if $A$ is a $3$-term arithmetic progression.
\end{theorem}

\begin{proof} We first prove (i). Suppose that $r_1=1$.
By Lemma \ref{lem1}, there exists a $(\mathbf{r}, h)$-path from
$V$ to $V'$: \begin{equation}\label{xyz} V=V_0\rightarrow
V_1\rightarrow \cdots\rightarrow V_t=V'.\end{equation} Let
$X=(x_0, x_1, x_2)\rightarrow Y$ be a $(\mathbf{r}, h)$-path. If
$x_1=0$, then $Y=(x_0-1,1,x_2)$. If $x_1=1$, then
$Y=(x_0,0,x_2+1)$. That is, $Y$ is uniquely determined by $X$.
Hence,  the $(\mathbf{r}, h)$-path \eqref{xyz} is uniquely
determined by $V$ and $V'$. For any $W\in R(\mathbf{r},h)$, by
Lemma \ref{lem1}, there exists a $(\mathbf{r}, h)$-path from $V$
to $W$ and a $(\mathbf{r}, h)$-path $W$ to $V'$. Since \eqref{xyz}
is unique, we have $W\in \{V_0,V_1,\ldots,V_t\}$. Thus, by the
definition of $h^{(\mathbf{r})}A$, $\phi_A (V_i) <\phi_A (V_{i+1})
(0\le i\le t-1)$ and Lemma \ref{lem0}, we have
\begin{eqnarray*}
|h^{(\mathbf{r})}A|&=&|\{ \phi_A (X) : X\in R(\mathbf{r},h)\} |\\
&=&|\{ \phi_A (V_i) : i=0,\dots , t\} |\\
&=& t+1=S(V')-S(V)+1=L(\mathbf{r}, h).\end{eqnarray*}

Next we shall prove (ii). If $A$ is a $3$-term arithmetic
progression, without loss of generality, we may assume that
$A=\{0,1,2\}$. By the proof of Theorem \ref{thm1}, we have
$|h^{(\mathbf{r})}A|=L(\mathbf{r}, h)$.

Conversely, suppose that $r_1\ge 2$ and
$|h^{(\mathbf{r})}A|=L(\mathbf{r}, h)$.

Since $2\le h\le r_0+r_1+r_2 -2$, there exists $ (x_0, x_1,
x_2)\in R(\mathbf{r},h)$ such that
$$1\le x_0\le r_0,\quad 1\le x_1\le r_1-1,\quad 0\le x_2\le r_2-1.$$
Then
$$ (x_0, x_1, x_2)\rightarrow (x_0-1, x_1+1, x_2)\rightarrow (x_0-1, x_1,
x_2+1)$$ and
$$ (x_0, x_1, x_2)\rightarrow (x_0, x_1-1, x_2+1)\rightarrow (x_0-1, x_1,
x_2+1)$$ are two different $(\mathbf{r}, h)$-paths. By Lemma
\ref{lem2}, we have
$$\phi_A ((x_0-1, x_1+1, x_2)) = \phi_A ((x_0, x_1-1, x_2+1)).$$ This implies that
$a_1-a_0=a_2-a_1$. Therefore, $A$ is a $3$-term arithmetic
progression.\end{proof}

\begin{theorem}\label{thm4} Let $A=\{ a_0<a_1<a_2<a_3\} $ be a set of integers
and $\mathbf{r}= (r_0,r_1,r_2,r_3)$ be an ordered $4$-tuple of
positive integers. Suppose that $h$ is an integer with $2\le h\le
r_0+r_1+r_2+r_3 -2$. Then

(i) for $r_1=r_2=1$, we have $|h^{(\mathbf{r})}A|=L(\mathbf{r},
h)$ if and only if $a_1-a_0=a_3-a_2$;

(ii) for $r_1\ge 2$ or $r_2\ge 2$, we have
$|h^{(\mathbf{r})}A|=L(\mathbf{r}, h)$ if and only if $A$ is a
$4$-term arithmetic progression.
\end{theorem}

\begin{proof}  Suppose that $|h^{(\mathbf{r})}A|=L(\mathbf{r}, h)$.

 Since $2\le h\le r_0+r_1+r_2+r_3 -2$, there exists $ (x_0, x_1,
x_2, x_3)\in R(\mathbf{r},h)$ such that
$$1\le x_0\le r_0,\quad  0\le x_1\le r_1-1,\quad 1\le x_2\le r_2, \quad 0\le x_3\le r_3-1.$$
Then
$$ (x_0, x_1, x_2,x_3)\rightarrow (x_0-1, x_1+1, x_2, x_3)\rightarrow (x_0-1, x_1+1,
x_2-1, x_3+1)$$ and
$$ (x_0, x_1, x_2, x_3)\rightarrow (x_0, x_1, x_2-1, x_3+1)\rightarrow (x_0-1, x_1+1,
x_2-1, x_3+1)$$ are two different $(\mathbf{r}, h)$-paths. By
Lemma \ref{lem2}, we have
$$\phi_A ((x_0-1, x_1+1, x_2, x_3)) = \phi_A ((x_0, x_1, x_2-1, x_3+1)).$$ This implies that
\begin{equation}\label{v} a_1-a_0=a_3-a_2.\end{equation}

We first prove (i).

It is enough to prove that if $r_1=r_2=1$ and $a_1-a_0=a_3-a_2$,
then $|h^{(\mathbf{r})}A|=L(\mathbf{r}, h)$.

 By Lemma \ref{lem1}, there exists a
$(\mathbf{r}, h)$-path from $V$ to $V'$
\begin{equation}\label{xys} V=V_0\rightarrow V_1\rightarrow
\cdots\rightarrow V_s=V'.\end{equation}  Suppose that
\begin{equation}\label{xyt} V=W_0\rightarrow W_1\rightarrow
\cdots\rightarrow W_t=V'\end{equation} is also a $(\mathbf{r},
h)$-path from $V$ to $V'$. By Lemma \ref{lem0}, we have $s=t$. Now
we prove that $\phi_A (V_i)=\phi_A (W_i) (0\le i\le s)$. In order
to prove this, we prove the following stronger result: for $0\le
i<s$, if $V_i=W_i$, then either $V_{i+1}=W_{i+1}$ or $V_{i+2}=
W_{i+2}$ and $\phi_A (V_{i+1})= \phi_A (W_{i+1})$.

Suppose that $0\le i<s$ and $V_i=W_i= (v_{i,0}, v_{i,1}, v_{i,2},
v_{i,3})$.

{\bf Case 1:} $v_{i,1}=v_{i,2}=0$. Then, by the definition of
step, we have
$$V_{i+1}=(v_{i,0}-1, v_{i,1}+1, v_{i,2},
v_{i,3}) =W_{i+1}.$$

{\bf Case 2:} $v_{i,1}=v_{i,2}=1$. Then, by the definition of
step and $r_1=r_2=1$, we have
$$V_{i+1}=(v_{i,0}, v_{i,1}, v_{i,2}-1,
v_{i,3}+1) =W_{i+1}.$$

{\bf Case 3:} $v_{i,1}=1, v_{i,2}=0$. Then, by the definition of
 step and $r_1=r_2=1$, we have
$$V_{i+1}=(v_{i,0}, v_{i,1}-1, v_{i,2}+1,
v_{i,3}+1) =W_{i+1}.$$

{\bf Case 4:} $v_{i,1}=0, v_{i,2}=1$. Then, by the definition of
 step and $r_1=r_2=1$, we have \begin{equation}\label{u}\{
V_{i+1},W_{i+1} \} \subseteq \{ (v_{i,0}-1, v_{i,1}+1, v_{i,2},
v_{i,3}), (v_{i,0}, v_{i,1}, v_{i,2}-1, v_{i,3}+1) \} .
\end{equation} Since
\begin{eqnarray*}&&\phi_A ((v_{i,0}-1, v_{i,1}+1, v_{i,2},
v_{i,3}))-\phi_A (V_i)\\
& =& a_1-a_0=a_3-a_2\\
& =& \phi_A ((v_{i,0}, v_{i,1}, v_{i,2}-1, v_{i,3}+1))-\phi_A
(V_i) ,\end{eqnarray*} we have $\phi_A (V_{i+1})= \phi_A
(W_{i+1})$.
 By \eqref{u}, the definition of
adjacency and $r_1=r_2=1$, we have
$$V_{i+2}=(v_{i,0}-1, v_{i,1}+1, v_{i,2}-1,
v_{i,3}+1)=W_{i+2}.$$

Thus, we have proved that for $0\le i<s$, if $V_i=W_i$, then
either $V_{i+1}=W_{i+1}$ or  $V_{i+2}= W_{i+2}$ and $\phi_A
(V_{i+1})= \phi_A (W_{i+1})$. It follows from $V_0=W_0$ and
$V_s=W_s$ that $\phi_A (V_i)=\phi_A (W_i) (0\le i\le s)$.

For any $W\in R(\mathbf{r},h)$, by Lemma \ref{lem1}, there exists
a $(\mathbf{r}, h)$-path from $V$ to $W$ and a $(\mathbf{r},
h)$-path $W$ to $V'$. By the above arguments, we have
$$\phi_A (W)\in \{ \phi_A (V_i) : 0\le i\le s\} .$$
Hence
$$h^{(\mathbf{r})}A = \{ \phi_A (X) : X\in R(\mathbf{r},h) \}
 =\{ \phi_A (V_i) : 0\le i\le s\} .$$
Therefore, by Lemma \ref{lem0},
$$|h^{(\mathbf{r})}A|=s+1=S(V')-S(V)+1=L(\mathbf{r}, h).$$

Now we prove (ii).

If $A$ is a $4$-term arithmetic progression, without loss of
generality, we may assume that $A=\{0,1,2,3\}$. By the proof of
Theorem \ref{thm1}, we have $|h^{(\mathbf{r})}A|=L(\mathbf{r},h)$.

Conversely, we suppose that $|h^{(\mathbf{r})}A|=L(\mathbf{r},h)$
and $r_1\ge 2$ or $r_2\ge 2$. By \eqref{v}, it is enough to prove
that $a_2-a_1=a_1-a_0$ or $a_2-a_1=a_3-a_2$.

{\bf Case 1:}  $r_1\ge 2$. Since $2\le h\le r_0+r_1+r_2+r_3 -2$,
there exists $Y=(y_0, y_1, y_2, y_3)\in R(\mathbf{r},h)$ such that
$$1\le y_0\le r_0,\quad  1\le y_1\le r_1-1,\quad 0\le y_2\le r_2-1, \quad 0\le y_3\le r_3.$$
Then
$$ (y_0, y_1, y_2,y_3)\rightarrow (y_0-1, y_1+1, y_2, y_3)\rightarrow (y_0-1, y_1,
y_2+1, y_3)$$ and
$$ (y_0, y_1, y_2, y_3)\rightarrow (y_0, y_1-1, y_2+1, y_3)\rightarrow (y_0-1, y_1,
y_2+1, y_3)$$ are two different $(\mathbf{r}, h)$-paths. By Lemma
\ref{lem2}, we have
$$\phi_A ((y_0-1, y_1+1, y_2, y_3)) = \phi_A ((y_0, y_1-1, y_2+1, y_3)).$$ This implies that
$a_1-a_0=a_2-a_1$.

{\bf Case 2:}  $r_2\ge 2$. Since $2\le h\le r_0+r_1+r_2+r_3 -2$,
there exists $Z=(z_0, z_1, z_2, z_3)\in R(\mathbf{r},h)$ such that
$$0\le z_0\le r_0,\quad  1\le z_1\le r_1,\quad 1\le z_2\le r_2-1, \quad 0\le z_3\le r_3-1.$$
Then
$$ (z_0, z_1, z_2,z_3)\rightarrow (z_0, z_1-1, z_2+1, z_3)\rightarrow (z_0, z_1-1,
z_2, z_3+1)$$ and
$$ (z_0, z_1, z_2, z_3)\rightarrow (z_0, z_1, z_2-1, z_3+1)\rightarrow (z_0, z_1-1,
z_2, z_3+1)$$ are two different $(\mathbf{r}, h)$-paths. By Lemma
\ref{lem2}, we have
$$\phi_A ((z_0, z_1-1, z_2+1, z_3)) = \phi_A ((z_0, z_1, z_2-1, z_3+1)).$$ This implies that
$a_2-a_1=a_3-a_2$.

Therefore, $A$ is a $4$-term arithmetic progression.
\end{proof}

\end{document}